\documentclass[12pt]{article}

\usepackage{marginnote}
\usepackage{amsmath}
\usepackage{amssymb}
\usepackage{amsthm}
\usepackage[latin1]{inputenc}
\usepackage{pxfonts}
\usepackage{txfonts}

\usepackage{graphicx}
\usepackage{epstopdf}

\usepackage{ifthen}
\usepackage{color}

\newtheorem{teo}{Theorem}[section]

\newtheorem{Lemma}{Lemma}[section]

\newtheorem{Proposition}{Proposition}[section]

\begin{document}

\title{ A semilinear elliptic equation with competing powers and a radial potential }
\author{Monica Musso
\,and\, Juliana Pimentel
}


\maketitle

\begin{abstract}

\noindent We verify the existence of radial positive solutions for the semi-linear equation
$$	-\,\Delta u=u^{p}\,-\,V(y)\,u^{q},\,\quad\quad u>0,\quad\quad\mbox{ in }\mathbb{R}^N$$
where  $N\geq 3$, $p$ is close  to $p^*:=(N+2)/(N-2)$, and $V$ is a radial smooth potential. If $q$ is super-critical, namely $q>p^*$,  we prove that this Problem has a radial solution behaving like a super-position of bubbles blowing-up at the origin with different rates of concentration, provided $V(0)<0$. On the other hand, if  $N/(N-2)<q<p^*$, we prove that this Problem has a radial solution behaving like a super-position of {\it flat}  bubbles  with different rates of concentration, provided $\lim_{r \to \infty} V(r) <0$.




\end{abstract}


\section{Introduction}\label{intro}

Let $N \geq 3$ and consider
\begin{equation}\label{eq1}
	-\,\Delta u=u^{p}\,-\,V(y)\,u^{q},\,\quad\quad u>0,\quad\quad\mbox{ in }\mathbb{R}^N\\
\end{equation}
where $V\in L^\infty(\mathbb{R}^N)$, $N\geq 3$, $q>p^s$, $p>p^*$, with
$$p^s=\frac{N}{N-2} ,\quad \quad  p^*=\frac{N+2}{N-2}.$$
In this paper, we are interested to the case $p$ slightly super-critical,
\begin{equation}
\begin{cases}\label{eq2}
	-\,\Delta u=u^{p^*+\epsilon}\,-\,V(y)\,u^{q},\, \mbox{ in }\mathbb{R}^N\\
	u(y)\rightarrow 0, \mbox{ as } |y|\rightarrow \infty\\
	\end{cases}
\end{equation}
where $\epsilon >0$.

For $q=1$,  problem \eqref{eq2} was treated in the critical case ($\epsilon$=0) in \cite{BC90} and in the sub-critical case ($\epsilon<0$) in \cite{DN86} . The super-critical analogue ($\epsilon>0$) was addressed in \cite{MMP05}, where it was proved the existence of a radial positive solution to \eqref{eq2} when $V$ is a radial smooth function with $V(0)<0$. A  previous construction can also be  found in \cite{MP03}.

In \cite{BFP00}, the authors consider problem \eqref{eq2} for any fixed $q$ satisfying $p^s<q<p^*$. It was proved the existence of an increasing number of rapidly decaying ground states, that is, solutions $u$ of \eqref{eq2} such that $\lim_{|x|\rightarrow \infty} u(x)=0$. The result  in \cite{BFP00}
is obtained  via tools in geometrical dynamical systems. The same equation was also treated in \cite{C08} and in \cite{DG16}, using a different approach, which also provided precise asymptotics for the solutions.
In bounded domains, the class of radial solutions behaving like a super-position of spikes was treated in the setting of super-critical exponents, in \cite{PDM03,PMP05}.

Let us now consider problem \eqref{eq2} in the super-critical case ($\epsilon>0$). In the case of a single power, i.e., $p^*+\epsilon=q$, and when $V(y) \equiv 1$, equation \eqref{eq2} is equivalent to
\begin{equation}
\begin{cases}\label{eq3}
	\Delta u+u^{p^*}=0,\, \mbox{ in }\mathbb{R}^N\\
	u>0,\, \mbox{ in }\mathbb{R}^N\\
		\end{cases}
\end{equation}
if we let $\epsilon$ goes to zero. It is well known that all bounded solutions of \eqref{eq3} are of the form
$$w_{\lambda,\xi}(x)=\gamma_N\left(  \frac{\lambda}{\lambda^2+|y-\xi|^2} \right)^{\frac{N-2}{2}},\qquad \gamma_N= \left( N(N-2) \right)^{\frac{N-2}{4}}$$
where $\lambda$ is a positive parameter and $\xi\in\mathbb{R}^N$, \cite{A76, T76,CGS89}. These functions are known in the literature as {\it bubbles}.

We want to prove the existence of a solution whose shape resembles a super-position of bubbles around the origin $0$ with different blow-up orders. This class of concentration phenomena is known as \textit{bubble-tower}. In the setting of semilinear elliptic equations with radial symmetry, these solutions were detected in a few situations, as we can see for instance in \cite{PDM03,CL99,CP06,MMP05,C08}.

Bubble-towers highly concentrated around the origin exist for
\eqref{eq2} under the assumption that  and $V(0)<0$. This is the content of our first result.

\begin{teo}\label{teo1}
Let $N\geq 3$ and $p^s<q<p^*<p$. Assume that $V\in L^\infty(\mathbb{R}^N)$ and $V(0)<0$. Then for every integer $k\geq 1$ there exists $\epsilon_k >0$ such that, for any $\epsilon \in (0, \epsilon_k )$, a solution $u_\epsilon$ of \eqref{eq2} exists and it has the form
\begin{equation}
u_\epsilon (x)=\gamma_N\sum_{j=1}^{k} \left(  \frac{1}{1+{\alpha_j}^{\frac{4}{N-2}}\epsilon^{-(j-1+\frac{1}{p^*-q})^\frac{4}{N-2}} |x|^2 }    \right)^{\frac{N-2}{2}} \alpha_j \epsilon^{-\left(  j-1+\frac{1}{p^*-q} \right)} (1+o(1))
\end{equation}
with $o(1)\rightarrow 0$ uniformly on compact sets of $\mathbb{R}^N$, as $\epsilon \rightarrow 0$. The constants $\alpha_i$ have explicit expression and depend only on $k, N, q$ and $V(0)$,
\begin{equation}
\alpha_j= \left[ - \frac{a_5 V(0)(p^*-q)}{a_3 k} \right]^{\frac{1}{p^*-q}} \left(  \frac{a_2}{a_3}  \right)^{j-1} \frac{(k-j)!}{(k-1)!}\quad j=1,\cdots, k,
\end{equation}
while $a_2$, $a_3$, $a_5$ are the positive constants defined in \eqref{cc1}.
\end{teo}

\medskip
Also in the case in which  $p^s<p^*<p<q$ bubble-towers do exist, but they are of a different nature, and their existence depends on the behavior of the potential $V$ at infinity.

\begin{teo}\label{teo2}
Let $N\geq 3$ and $q>p^*$. Assume that $V\in L^\infty(\mathbb{R}^N)$ and $V_\infty:=\lim_{|x|\rightarrow\infty}V(x)<0$. Then for every integer $k\geq 1$ there exists $\epsilon_k >0$ such that, for any $\epsilon \in (0, \epsilon_k )$, a solution $u_\epsilon$ of \eqref{eq2} exists and it has  the form
\begin{equation}\label{ii}
\hat{u}_\epsilon (x)=\gamma_N\sum_{j=1}^{k} \left(  \frac{1}{1+{\hat{\alpha}_j}^{\frac{4}{N-2}}\epsilon^{(j-1+\frac{1}{q-p^*})^\frac{4}{N-2}} |x|^2 }    \right)^{\frac{N-2}{2}} \hat{\alpha}_j \epsilon^{\left(  j-1+\frac{1}{q-p^*} \right)} (1+o(1))
\end{equation}
with $o(1)\rightarrow 0$ uniformly on compact sets of $\mathbb{R}^N$, as $\epsilon \rightarrow 0$. The constants $\hat{\alpha}_i$ have explicit expression and depend only on $k, N, q$ and $V_\infty$,
\begin{equation}
\hat{\alpha}_j= \left[  \frac{\hat{a}_5 V_\infty(p^*-q)}{a_3 k} \right]^{\frac{1}{p^*-q}} \left(  \frac{a_2}{a_3}  \right)^{j-1} \frac{(k-j)!}{(k-1)!}, \quad j=1,\cdots, k,
\end{equation}
while $a_2$, $a_3$, $\hat{a}_5$ are the positive constants defined in \eqref{cc1} and \eqref{cc2}.
\end{teo}

The bubble-tower in \eqref{ii} describes a superposition of $k$ {\it flat bubbles}.

\medskip
In order to prove our results, we start by reducing the problem to a non-autonomous ordinary differential equation, using the so-called Emden-Fowler transformation, \cite{F31}. Then we perform a Lyapunov-Schmidt reduction, as in  \cite{FW86}, to reduce the procedure of construction of solutions to a finite-dimensional variational problem.

The paper is organized as follows. In Section $2$, we provide an asymptotic expansion of the energy functional associated to the ODE problem. The finite dimensional reduction argument is discussed in Section $3$. We prove Theorems \ref{teo1} and \ref{teo2} in Sections $4$ and $5$, respectively.

\section{The energy asymptotic expansion}\label{intro}

Since we are seeking for a solution $u$ of \eqref{eq2} with fast decay, we can assume that $u$ is radial around the origin. Then we arrive at the following equivalent problem

\begin{equation}
\begin{cases}\label{eq3}
	u''(r)+ \frac{N-1}{r} u'(r)+ u^{p^*+\epsilon}(r)- V(r) u^q(r)=0\\
	u(r)\rightarrow 0, \mbox{ as } r \rightarrow \infty\\
\end{cases}
\end{equation}
By introducing the so-called Emden-Fowler transformation,
\begin{equation}
v(x)=r^{\frac{2}{p^*-1}} u(r),\quad \mbox{ with } r=e^{-\frac{p^*-1}{2}x},
\end{equation}
for $x\in\mathbb{R}$, the problem \eqref{eq3} becomes
\begin{equation}
\begin{cases}\label{eq4}
	v''(x)-v(x)+\beta \left[   e^{\epsilon x} v^{p^*+\epsilon}(x)- V\left( e^{-\frac{p^*-1}{2}x} \right)  e^{-(p^*-q)x}  v^q(x)    \right] =0,\\
	
	0<v(x)\rightarrow 0, \mbox{ as } |x| \rightarrow \infty\\
\end{cases}
\end{equation}
in $\mathbb{R}$, where $\beta=\left( \frac{2}{N-2}    \right)^2$. We henceforth denote $\omega(x)=V\left( e^{-\frac{p^*-1}{2}x} \right)$.

The energy functional related to \eqref{eq4} is
\begin{equation}\label{eq5}
E_\epsilon(\psi)=I_\epsilon (\psi)+\frac{\beta}{q+1}\int_{\mathbb{R}} \omega(x) e^{-(p^*-q)x} |\psi|^{q+1} dx
\end{equation}
where
$$I_\epsilon (\psi)=   \frac{1}{2}\int_{\mathbb{R}} (|\psi'|^2 + |\psi|^2) dx -  \frac{\beta}{p^*+\epsilon+1}\int_{\mathbb{R}}  e^{\epsilon x} |\psi|^{p^*+\epsilon+1} dx.$$

Let us consider the positive radial solution of
\begin{equation}
\label{eq6}
	\Delta w+ w^{p^*}=0, \quad
	w(0)=\gamma_N
\end{equation}
given by $w(r)=\gamma_N\left(   \frac{1}{1+r^2} \right)^{\frac{N-2}{2}}$.  Now we set $U$ to be the Emden-Fowler transformation of $w$
\begin{equation}\label{eq7}
U(x)=\gamma_N e^{-x} \left(  1+e^{-(p^*-1)x} \right)^{-\frac{N-2}{2}}.
\end{equation}
Then $U$ satisfies
\begin{equation}
\label{eq8}
	U''-U+\beta U^{p^*}=0, \quad
	0<U(x)\rightarrow 0, \mbox{ as } |x| \rightarrow \infty
\end{equation}
It is then natural to look for a solution of \eqref{eq4} of the form
$$v(x)=\sum_{i=1}^{k} U(x-\xi_i) +\phi (x)$$
for certain choice of points $0<\xi_1<\xi_2<\cdots<\xi_k$ and $\phi$ is small. We set
\begin{equation}\label{eq8}
U_i(x)=U(x-\xi_i),\quad \bar{U}=\sum_{i=1}^k U_i(x).
\end{equation}
and choose the points $\xi_i$ as follows:
\begin{align}\label{eq7}
&\xi_1= -\frac{1}{p^*-q} \log \epsilon - \log \Lambda_1\\
\nonumber \xi_{i+1}-\xi_i=- &\log \epsilon -\log \Lambda_{i+1}, \quad i=1,\cdots, k-1
\end{align}
where the $\Lambda_i$'s are positive parameters. This choice of the $\xi_i$'s turns out to be convenient in the proof of the following asymptotic expansion of $E_\epsilon(\bar{U})$. We set $\Lambda=(\Lambda_1,\Lambda_2,\cdots,\Lambda_k)$.

\begin{Lemma}\label{lem1}
Let $N\geq 3$, $\delta>0$ fixed, $k\in \mathbb{N}$. Assume that
\begin{equation}\label{eq9}
\delta<\Lambda_i <\delta^{-1},\quad i=1,2,\cdots, k.
\end{equation}
Then there exist positive numbers $a_1$, $i=1,\cdots, 5$, depending on $N$, $p$ and $q$, such that
$$E_\epsilon (\bar{U})= k a_1 + \epsilon \Psi_k(\Lambda)  +  k\epsilon \beta a_4 + \epsilon \theta_\epsilon (\Lambda) - \frac{a_3 k}{2(p^*-p)}((1-k)(p^*-q)-2) \epsilon \log \epsilon $$
where
\begin{equation}\label{PSI}
\Psi_k(\Lambda) =  a_3 k \log \Lambda_1 + a_5 V(0) \Lambda_{1}^{(p^*-q)} + \sum_{i=1}^{k} [(k-i+1) a_3 \log \Lambda_i - a_2 \Lambda_i  ]
\end{equation}
and $\theta_\epsilon(\Lambda)\rightarrow 0$ as $\epsilon \rightarrow 0$ uniformly in $C^1$-sense on the set of $\Lambda_i$'s satisfying \eqref{eq9}.
\end{Lemma}

\begin{proof}
We estimate

\begin{align*}
I_\epsilon(\bar{U})&=\frac{1}{2} \int_{\mathbb{R}}  \left(  | \bar{U}' |^2 + | \bar{U}  |^2 \right)  dx - \frac{\beta}{p^*+\epsilon+1} \int_{\mathbb{R}} e^{\epsilon x} |\bar{U}|^{p^*+\epsilon+1}dx\\
&=I_0(\bar{U}) -\frac{\beta}{p^*+1} \int_{\mathbb{R}} (e^{\epsilon x}-1) |\bar{U}|^{p^*+\epsilon+1} dx\\
&\quad+ \left(  \frac{1}{p^*+1}- \frac{1}{p^*+1+\epsilon}  \right) \beta \int_{\mathbb{R}} e^{\epsilon x} |\bar{U}|^{p^*+\epsilon+ 1} dx\\
&\quad +\frac{\beta}{p^*+1}  \int_{\mathbb{R}}  \left(     |\bar{U}|^{p^*+ 1} -|\bar{U}|^{p^*+\epsilon+ 1}  \right)dx\\
& = I_0(\bar{U}) -\frac{\beta}{p^*+1} \int_{\mathbb{R}} (e^{\epsilon x}-1) |\bar{U}|^{p^*+\epsilon+1} + A_\epsilon.
\end{align*}
where
\begin{align*}
A_\epsilon=&\left(  \frac{1}{p^*+1}- \frac{1}{p^*+1+\epsilon}  \right) \beta \int_{\mathbb{R}} e^{\epsilon x} |\bar{U}|^{p^*+\epsilon+ 1} dx\\
&\quad +\frac{\beta}{p^*+1}  \int_{\mathbb{R}}  \left(     |\bar{U}|^{p^*+ 1} -|\bar{U}|^{p^*+\epsilon+ 1}  \right)dx
\end{align*}
As in \cite{PDM03}, we can prove that
\begin{equation*}
A_\epsilon= k\epsilon \beta \left(    \frac{1}{(1+p^*)^2}  \int_{\mathbb{R}} |U|^{p^*+1}dx  - \frac{1}{(1+p^*)}  \int_{\mathbb{R}} |U|^{p^*+1} \log U dx \right) +o(\epsilon).
\end{equation*}
Also, by reasoning in a similar manner we have

\begin{align*}
\int_{\mathbb{R}} (e^{\epsilon x}-1) &|\bar{U}|^{p^*+\epsilon+1} dx= \epsilon \int_{\mathbb{R}}  x |\bar{U}|^{p^*+\epsilon+1}dx+ o(\epsilon)\\
& =\epsilon \left(  \sum_{l=1}^k \xi_l \right) \int_{\mathbb{R}} U^{p^*+1}dy +o(\epsilon).
\end{align*}
and
\begin{equation*}
I_0(\bar{U})=k I_0(U)-\beta C_N \int_{\mathbb{R}} U^{p^*} e^x dx \left(  \sum_{l=2}^k e^{\xi_l-\xi_{l-1}} \right)+o(\epsilon).
\end{equation*}

Now we need to evaluate $\int_{\mathbb{R}} \omega(x)e^{-(p^*-q)x} |\bar{U}|^{q+1}dx$. By following the argument in \cite{MMP05} et al and using our choice of $\xi_l's$, we have
\begin{align*}
\int_{\mathbb{R}} \omega(x) e^{-(p^*-q)x } |\bar{U}|^{q+1} dx & = \sum_{i=1}^{k} \int_{\mathbb{R}} \omega(x) e^{-(p^*-q)x} |U_i|^{q+1}dx+o(\epsilon)\\
& = \int_{\mathbb{R}} \omega(x) e^{-(p^*-q)x } |U_1|^{q+1} dx + o(\epsilon).
\end{align*}
On the other hand, the following holds
\begin{align*}
\int_{\mathbb{R}} \omega(x) e^{-(p^*-q)x}& |U_1|^{q+1} dx =  e^{-(p^*-q)\xi_1 } \int_{\mathbb{R}} \omega(x+\xi_1) e^{-(p^*-q)x } |U_1 (x+\xi_1)|^{q+1}dx\\
&= e^{-(p^*-q)\xi_1 } \int_{\mathbb{R}} \omega(x+\xi_1) e^{-(p^*-q)x } |U|^{q+1}dx\\
&= e^{-(p^*-q)\xi_1 } V(0) \int_{\mathbb{R}} e^{-(p^*-q)x } |U|^{q+1}dx +o(1).
\end{align*}

We thus have the following
\begin{align*}
E_\epsilon(\bar{U})= & I_\epsilon (\bar{U}) +\frac{\beta}{q+1}\int_{\mathbb{R}} \omega(x) e^{-(p^*-q)x}|\bar{U}|^{q+1}dx\\
= & I_0 (\bar{U}) -\frac{\beta}{p^*+1}\int_{\mathbb{R}}  (e^{\epsilon x}-1) |\bar{U}|^{p^*+\epsilon+1}dx+ A_\epsilon +\\
& + \frac{\beta}{q+1}\int_{\mathbb{R}} \omega(x) e^{-(p^*-q)x}|\bar{U}|^{q+1}dx \\
= & k I_0(U)  -  \beta C_N \int_{\mathbb{R}} U^{p^*} e^x dx   \left(  \sum_{l=2}^k  e^{\xi_l-\xi_{l-1}} \right)  \\
&-\frac{\beta}{p^*+1} \left(  \epsilon   \left(  \sum_{l=1}^k  \xi_l \right) \int_{\mathbb{R}} U^{p^*+1}  dy \right)  \\
& + k\epsilon \beta \left(    \frac{1}{(1+p^*)^2}  \int_{\mathbb{R}} |U|^{p^*+1}dx -  \frac{1}{1+p^*}  \int_{\mathbb{R}} |U|^{p^*+1} \log U dx   \right)+\\
&  + \frac{\beta}{q+1} \left(  e^{-(p^*-q)\xi_1} V(0)  \int_{\mathbb{R}} e^{-(p^*-q)x}  |U|^{q+1} dx \right) +o(\epsilon )
\end{align*}
which lead us to the following expression:
\begin{equation*}
E_\epsilon(\bar{U})= k a_1 - a_2 \sum_{l=2}^{k}   e^{-(\xi_l-\xi_{l-1})}    - \epsilon a_3 \left(   \sum_{i=1}^{k} \xi_i  \right) + k\epsilon\beta a_4 + a_5 V(0) e^{-(p^*-q)\xi_1}+o(\epsilon).
\end{equation*}
By using our choice of $\xi_i$'s
\begin{equation*}
E_\epsilon(\bar{U})= k a_1  + \epsilon \Psi_k(\Lambda)  -    \frac{a_3 k}{2(p^*-q)} \left(   (1-k)(p^*-q) -2  \right) \epsilon \log \epsilon +k\epsilon\beta a_4  \epsilon + o(\epsilon),
\end{equation*}
where $\Psi_k(\Lambda)$ is given by \eqref{PSI} and the constants $a_i$, $i=1,\cdots, 5$, are explicitly expressed as follows
\begin{equation}\label{cc1}
\begin{cases}
a_1=I_0(U), \quad
a_2=\beta C_N \int_{\mathbb{R}} U^{p^*}(x) e^x dx, \quad
a_3=\frac{\beta}{p^* +1} \int_{\mathbb{R}} U^{p^*+1}(x)  dx\\ \\
a_4=\frac{1}{(p^* +1)^2} \int_{\mathbb{R}} U^{p^*+1}(x)  dx- \frac{1}{p^*+1} \int_{\mathbb{R}} U^{p^*+1}(x) \log U (x)  dx\\ \\
a_5=\frac{\beta}{q +1}  \int_{\mathbb{R}}  e^{-(p^*-q)x}   U^{q+1}(x)  dx
\end{cases}
\end{equation}

Notice that the term $o(\epsilon)$ in the above expression of $E_\epsilon (\bar{U})$ is uniform in the set of the $\Lambda_i's$ satisfying \eqref{eq9}. A similar computation shows that differentiation with respect to the $\Lambda_i's$ leaves the term $o(\epsilon)$ of the same order in the $C^1$-sense.

\end{proof}

\section{The finite dimensional reduction}\label{sec3}

We consider again points $0<\xi_1<\xi_2<\cdots<\xi_k$ which are for now arbitrary and define
$$Z_i(x)=U'_i(x),\quad i=1,\cdots, k.$$
Next we consider the problem of finding a function $\phi$ for which there are constants $c_i$, $i=1,\cdots,k$, such that, in $\mathbb{R}$
\begin{equation}
\begin{cases}\label{eq10}
	\sum_{i=1}^{k}c_i Z_i= -(\bar{U}+ \phi)''+ (\bar{U}+\phi) - \beta \left[  e^{\epsilon x} (\bar{U}+\phi)_+^{p^*+\epsilon} - \omega (x) e^{-(p^*-q)x} (\bar{U}+\phi)^q   \right] \\ \\
\phi (x)\rightarrow 0, \, |x| \rightarrow \infty, \\  \\  \int_{\mathbb{R}}Z_i \phi dx =0, \quad i=1,\cdots, k

\end{cases}
\end{equation}
Let us consider the linearized operator around $\bar{U}$
$$\mathcal{L}_\epsilon \phi = -\phi'' + \phi - \beta \left[   (p^* +\epsilon) e^{\epsilon x} \bar{U}^{p^*+\epsilon-1} - q \omega (x)  e^{-(p^*- q)x} \bar{U}^{q-1}  \right] \phi. $$
Then \eqref{eq10} can be rewritten as
\begin{equation}
\begin{cases}\label{eq11}
	\mathcal{L}_\epsilon \phi = N_\epsilon^1 (\phi) +N_\epsilon^2 (\phi) +R_\epsilon + \sum_{i=1}^{k}c_i Z_i  \\ \\
\phi (x)\rightarrow 0, \, |x| \rightarrow \infty, \\  \\  \int_{\mathbb{R}}Z_i \phi dx =0, \quad i=1,\cdots, k

\end{cases}
\end{equation}
where
$$N_\epsilon^1=\beta e^{\epsilon x} [(\bar{U}+\phi)^{p^*+\epsilon} -\bar{U}^{p^*+ \epsilon}-(p^*+\epsilon)\bar{U}^{p^*+\epsilon-1}\phi ],$$
$$N_\epsilon^2= -\beta \omega(x) e^{-(p^*-q) x} [(\bar{U}+\phi)^{q} -\bar{U}^{q}- q \bar{U}^{q-1}\phi ],$$
$$R_\epsilon= \sum_{i=1}^k  U_i^{p^*} + \beta e^{\epsilon x} \bar{U}^{p^*+\epsilon} - \beta \omega (x) e^{-(p^*-q)x} \bar{U}^{q}.$$

Next we prove that \eqref{eq11} has a solution for certain choice of $\xi_i$. In order to do that we first analyze its linear part, i.e., we consider the problem of, given a function $h$, finding $\phi$ such that
\begin{equation}
\begin{cases}\label{eq12}
	\mathcal{L}_\epsilon \phi = h + \sum_{i=1}^{k}c_i Z_i  \\ \\
\phi (x)\rightarrow 0, \, |x| \rightarrow \infty, \\  \\  \int_{\mathbb{R}}Z_i \phi dx =0, \quad i=1,\cdots, k

\end{cases}
\end{equation}
In order to analyze invertibility properties of $\mathcal{L}_\epsilon$ under the orthogonality conditions, we introduce the following norm for function $\psi:\mathbb{R}\rightarrow\mathbb{R}$
$$\Vert  \psi  \Vert_\ast =\sup_{x\in\mathbb{R}} \left(  \sum_{i=1}^{k} e^{-\sigma | x-\xi | } \right)^{-1} |\psi (x)| $$
where $\sigma>0$ is a small constant to be fixed later.

The following result holds.

\begin{Proposition}\label{prop1}
There exist positive numbers $\epsilon_0$, $\delta_0$, $R_0$ such that if
\begin{equation}\label{cond_on_xi}
R_0<\xi_i, \quad \quad R_0<\min_{1\leq i <k}(\xi_{i+1}-\xi_i), \quad\quad \xi_k<\frac{\delta_0}{\epsilon}
\end{equation}
then for all $0<\epsilon<\epsilon_0$ and for all $h\in C(\mathbb{R})$ with $\vert h \vert_\ast <+\infty$, problem \eqref{eq12} has a unique solution $\psi=: T_\epsilon(h)$, such that
\begin{equation*}
\Vert  T_\epsilon (h) \Vert_\ast \leq C\Vert h \Vert_\ast, \quad \quad |c_i|\leq \Vert h \Vert_\ast.
\end{equation*}
\end{Proposition}

\begin{Lemma}\label{lem1}
Assume there is a sequence $\epsilon_n\rightarrow 0$ and points $\xi_i$'s satisfying $0<\xi_1^n<\cdots<\xi_k^n$ with
\begin{equation}
\xi_1^n\rightarrow \infty,  \quad \quad \min_{1\leq i < k} (\xi_{i+1}^n-\xi_i^n )\rightarrow \infty,\quad \quad \xi_k^n=o(\epsilon_n^{-1})
\end{equation}
such that for certain functions $\phi_n$ and $h_n$ with $\Vert h_n\Vert_\ast\rightarrow 0$, and scalars $c_i^n$, one has in $\mathbb{R}$
\begin{equation}
\begin{cases}\label{eq13}
	\mathcal{L}_{\epsilon_n} (\phi_n) = h_n + \sum_{i=1}^{k}c_i^n Z_i^n  \\ \\
\phi_n (x)\rightarrow 0, \, |x| \rightarrow \infty, \\  \\  \int_{\mathbb{R}}Z_i ^n\phi_n dx =0, \quad i=1,\cdots, k

\end{cases}
\end{equation}
with $Z_i^n(x)=U'(x-\xi_i^n)$. Then $\lim_{n\rightarrow \infty}\Vert  \phi_n \Vert_\ast=0$
\end{Lemma}

\begin{proof}
We first establish the weaker assertion that
$$\lim_{n\rightarrow \infty}  \Vert  \phi_n \Vert_\infty=0.$$
By contradiction, we may assume that $\Vert  \phi_n \Vert_\infty=1$. Testing \eqref{eq13}  against $Z_l^n$ and integrating by parts we get
$$\sum_{i=1}^k c_i^n \int_\mathbb{R}Z_i^nZ_l^n dx= \int_\mathbb{R} \mathcal{L}_{\epsilon_n} (Z_l^n) \phi_n dx -  \int_\mathbb{R}h_n Z_l^n dx.$$
This defines a linear system in the $c_i$'s which is "almost diagonal" as $n\rightarrow \infty$. Moreover, the assumptions made plus the fact that the $Z_l^n$ solves
$$-Z''+(1-p^* \beta U_l^{p^*-1}Z)=0$$
yields, after an application of dominated convergence, that $\lim_{n\rightarrow \infty}c_i^n=0.$ If we set $x_n\in\mathbb{R}^N$ such that $\phi_n(x_n)=1$, we can assume that there exists $i\in\{ 1,\cdots, k \}$ such that for $n$ large enough such that $| \xi_l^n- x_n |<R$ for some fixed $R>0$. We set $\tilde{\phi}_n=\phi_n(x+\xi_i^n)$. From \eqref{eq13}, we see that passing to a suitable subsequence, $\tilde{\phi}_n(x)$ converges uniformly over compacts to a nontrivial bounded solution $\tilde{\phi}$ of
$$-\tilde{\phi}''+\tilde{\phi}-\beta p^* U^{p^*}\tilde{\phi}=0, \quad \quad \mbox{ in } \mathbb{R}.$$
Hence for some $c\neq 0$, $\tilde{\phi}=c U'$. However the orthogonality condition passes to the limit as
$$0=\int_{\mathbb{R}}Z_l^n\phi_n\,\,\rightarrow\,\, c\int{\mathbb{R}}(U')^2$$
which is a contradiction. Then $\lim_{n\rightarrow \infty}  \Vert  \phi_n \Vert_\infty=0$.

Now, we observe that shows that \eqref{eq13} takes the form
\begin{equation}\label{eq14}
-\phi^{''}_n+\phi_n=g_n
\end{equation}
where
$$g_n=h_n+\sum_{i=1}^{k} c_i^n Z_i^n+\beta [(p^*+\epsilon_n)e^{\epsilon_n x}\bar{U}^{p^*+\epsilon_n-1}-q \omega (x)e^{-(p^*-q)x}\bar{U}^{q-1}]\phi_n.$$
We estimate $g_n$,
\begin{align*}
|g_n|  \, \leq  \, & \Vert h_n \Vert_* \left( \sum_{i=1}^{k}e^{-\sigma |x-\xi_i^n|}  \right)   +   c_{l}^n   \sum_{i=1}^n o  \left(  e^{- |x-\xi_i^n|}\right)+  \\
&+\Vert  \phi_n \Vert_\infty   \left( \sum_{i=1}^n o  \left(e^{- (p^*-1)|x-\xi_i^n|}\right)         + \sum_{i=1}^no\left(e^{- (2q-p^*-1)|x-\xi_i^n|}\right)   \right),
\end{align*}
since $V\in L^\infty(\mathbb{R})$. If $0<\sigma<\min\{  1, p^*-1, 2q-p^*-1 \}$, we have
$$ | g_n(x) | \leq \theta _n \sum_{i=1}^k e^{-\sigma | x-\xi_i |}=: \psi_n(x), $$
with $\theta_n\rightarrow 0$. We see the that the function $C\psi_n$, for $C>0$ sufficiently large, is a supersolution for \eqref{eq14}, so that $\phi_n\leq C\psi_n$. Similarly, we have $\phi_n\geq -C\psi_n$. Thus, the proof is concluded.
\end{proof}

Then the proof of Proposition \ref{prop1} follows from Lemma \ref{lem1} as in \cite{MMP05}.

Next we study some differentiability properties of $T_\epsilon$ on $\xi_i$. We write $\xi=(\xi_1,\cdots,\xi_k)$. We let $\mathcal{C}_*$ be the Banach space of all continuous $\psi$ defined in $\mathbb{R}$ satisfying $\Vert  \psi \Vert_*<\infty$, endowed with the norm $\Vert \cdot \Vert_*$. Also, let $\mathcal{L}(\mathcal{C_*})$ be the space of linear operators of $\mathcal{C}_ *$.

The following result can be established.

\begin{Proposition}\label{prop2}
Under the assumptions of the Proposition \ref{prop1}, consider the map $T_\epsilon(\xi)$ with values in $\mathcal{L}(\mathcal{C_*})$. Then $T_\epsilon$ is $C^1$ and
$$\Vert D_\xi T_\xi \Vert_{\mathcal{L}(\mathcal{C_*})}\leq C$$
uniformly on $\xi$ satisfying \eqref{cond_on_xi}, for some constant $C$.
\end{Proposition}
\begin{proof}
Fix $h\in \mathcal{C}_*$ and let $\phi=T_\epsilon(h)$ for $\epsilon<\epsilon_0$. Notice that $\phi$ satisfies \eqref{eq12} and the orthogonality conditions, for some constants $c_i$ uniquely determined. For $l\in \{ 1,\cdots, k \}$, if we define the constant $b_l$ as follows
$$b_l\int_\mathbb{R} |Z_l|^2=\int_\mathbb{R} \phi \partial_{\xi_l}Z_l.$$
By differentiating with respect to $\xi_l$ we obtain that
$$\partial_{\xi_l}\phi=T_\epsilon (f)+b_l Z_l$$
where
$$f=-b_l \mathcal{L}_\epsilon Z_l +c_l\partial_{\xi_l} Z_l +\beta[(p^*+\epsilon)e^{\epsilon x}(\partial_{\xi_l} \bar{U}^{p^*+\epsilon-1})- q \omega (x) e^{-(p^*-q)x} (\partial_{\xi_l} \bar{U}^{q-1})]\phi.$$
Moreover $\Vert f \Vert_*\leq C\Vert h\Vert_* $, $|b_l|\leq C\Vert \phi \Vert_*$ so that $\Vert \partial_{\xi_l} \phi \Vert \leq C \Vert h \Vert_*$. Besides $\partial_{\xi_l} \phi$ depends continuously on $\xi$ for this norm. Thus, the result follows.
\end{proof}

We are now ready to prove that \eqref{eq11} is uniquely solvable with respect to $\Vert \phi \Vert_*$. In order to do that we restrict the range of the parameters $\xi_i$'s  in a convenient way. We assume that, for a fixed $M>0$ large, the following conditions hold
\begin{equation}\label{eq15}
\log(M \epsilon)^{-1}< \min_{1\leq i<k} (\xi_{i+1}-\xi_{i}), \quad \quad \xi_k < k \log(M\epsilon)^{-1}.
\end{equation}
Then we can estimate $R_\epsilon$, $N_\epsilon^1+N_\epsilon^2$, and their derivatives, by direct calculation, as follows.

\begin{Lemma}\label{lem2}
If $\Vert  \phi \Vert_1\leq \frac{1}{2}\Vert  \bar{U} \Vert_1$ then
\begin{align*}
&\Vert  N_\epsilon (\phi) \Vert_*\leq C(  \Vert  \phi \Vert_*^{\min \{p^*,2\}} +\Vert  \phi \Vert_*^{\min \{2q-p^*,2\}} )\\
&\Vert  D_\phi N_\epsilon (\phi) \Vert_*\leq   C(  \Vert  \phi \Vert_*^{\min \{p^*-1,2\}} +\Vert  \phi \Vert_*^{\min \{2q-p^*-1,2\} })
\end{align*}
where $\Vert \phi \Vert_1:=\sup_{x\in\mathbb{R}} \left( \sum_{i=1}^k   e^{|x-\xi_i|}  \right)^{-1} |\phi(x)|   $ and $N_\epsilon(\phi)=N^1_\epsilon(\phi)+N^2_\epsilon(\phi)$.
In addition,  if \eqref{eq15} holds then
$$
\Vert  R_\epsilon  \Vert_*\leq C \epsilon^{\frac{1+\tau}{2}}  , \quad
\Vert  \partial_\xi R_\epsilon  \Vert_*\leq C \epsilon^{\frac{1+\tau}{2}}
$$
where $\tau>0$ is small.
\end{Lemma}

The next result allows for the reduction to a finite dimensional problem, as we will see in the next section. The proof is very similar to \cite[Proposition $3$]{MMP05} and we omit it here.


\begin{Proposition}\label{prop3}
Assume \eqref{eq15} holds. Then, for all $\epsilon$ small enough, there exists a unique solution $\phi=\phi(\xi)$ to problem \eqref{eq10}, which satisfies
\begin{equation*}
\Vert  \phi \Vert_* \leq C \epsilon^{\frac{1+\tau}{2}}.
\end{equation*}
Moreover, the map $\xi\mapsto \phi(\xi)$ is of class $C^1$ for the norm $\Vert \cdot\Vert_*$ and
$$\Vert  D_\xi \phi \Vert_*\leq  C \epsilon^{\frac{1+\tau}{2}}.$$
\end{Proposition}

\section{The finite dimensional variational problem}\label{sec4}

In this section we fix a large constant $M>0$ and assume the conditions \eqref{eq15} for $\xi$. Our problem is equivalent to that of finding $\xi_i$'s satisfying $c_i(\xi)=0$, for all $i=1,2,\cdots,k$. In this case, $v=\bar{U}+\phi$ is a solution for \eqref{eq4} satisfying the desired formula.

We consider the functional
$$\mathcal{I}_\epsilon (\xi)= E_\epsilon (\bar{U}+\phi),$$
where $\phi=\phi(\xi)$ is that of Proposition \ref{prop3} and $E_\epsilon$ is the energy functional defined in \eqref{eq5}. It is known that finding the desired $c_i$'s is equivalent to finding a critical point of $\mathcal{I}_\epsilon (\xi)$, see for instance \cite{MMP05}. That is, we need to find a point $\xi$ satisfying
\begin{equation}\label{eq16}
\nabla \mathcal{I}_\epsilon(\xi)=0 .
\end{equation}
In order to do that, the following expansion result will be crucial.

\begin{Lemma}
The following expansion holds
$$\mathcal{I}_\epsilon(\xi)=E_\epsilon (\bar{U})+o(\epsilon),$$
where $o(\epsilon)$ is uniform in the $C^1$-sense over all points $\xi$ satisfying \eqref{eq15}.
\end{Lemma}
\begin{proof}
First, notice that $D E_\epsilon (\bar{U}+\phi)[\phi] = 0$. It then follows from Taylor expansion that
\begin{align}\label{eq17}
E_\epsilon  (\bar{U}+\phi) - E_\epsilon(\bar{U}) &= \int_0^1 D^2 E_\epsilon (\bar{U}+ t \phi) [\phi^2] t dt\\
\nonumber& = \int_0^1    \int_{\mathbb{R}} [N_\epsilon(\phi)+ R_\epsilon]\phi  t dt\\
\nonumber& +  \int_0^1 \int_{\mathbb{R}} \beta (p^*+ \epsilon) e^{\epsilon x} [\bar{U}^{p^*+\epsilon -1} -\left(   \bar{U}+t\phi   \right)^{p^*+\epsilon -1}   ]  \phi^2 t dt\\
\nonumber&  - \int_0^1 \beta  q \int_{\mathbb{R}} \omega(x) e^{-(p^*-q)x} [\bar{U}^{q-1} -(\bar{U}+ t \phi)^{q-1} ] \phi^2      t dt .
\end{align}
Since $\Vert  \phi \Vert_* \leq C e^{\frac{1+\tau}{2}}$, from Lemma \ref{lem2}, we get
$$ \mathcal{I_\epsilon}(\xi)-E_\epsilon (\bar{U}) = o (e^{1+\tau})$$
uniformly on points satisfying \eqref{eq15}. Next we differentiate with respect to $\xi$ and get, from \eqref{eq17} that
\begin{align*}
D_\xi [\mathcal{I}_{\epsilon}(\xi) - E_\epsilon (\bar{U})] & = \int_0^1    \int_{\mathbb{R}} D_\xi [N_\epsilon(\phi)+ R_\epsilon]\phi  t dt\\
& +  \int_0^1 \int_{\mathbb{R}} \beta (p^*+ \epsilon) e^{\epsilon x} D_\xi [\bar{U}^{p^*+\epsilon -1} -\left(   \bar{U}+t\phi   \right)^{p^*+\epsilon -1}   ]  \phi^2 t dt\\
&  - \int_0^1 \beta  q \int_{\mathbb{R}} \omega(x) e^{-(p^*-q)x} D_\xi [\bar{U}^{q-1} -(\bar{U}+ t \phi)^{q-1} ] \phi^2      t dt .
\end{align*}
Using similar arguments as in Proposition \ref{prop2}, we find that
$$D_\xi [\mathcal{I}_\epsilon (\xi) -E_\epsilon (\bar{U})]=o(\epsilon^{1+\tau}).$$
Thus the result follows.
\end{proof}

In what follows we prove Theorem \ref{teo1}.

\paragraph{Proof of Theorem 1.1} Recall that
\begin{align*}
&\xi_1= -\frac{1}{p^*-q} \log \epsilon - \log \Lambda_1\\
\nonumber \xi_{i+1}-\xi_i=- &\log \epsilon -\log \Lambda_{i+1}, \quad i=1,\cdots, k-1
\end{align*}
where $\Lambda_i$'s are positive parameters. Thus, it is sufficient to find a critical point of
$$\Phi_\epsilon(\Lambda)= \epsilon^{-1} \mathcal{I}_\epsilon (\xi(\Lambda)).$$
Now, from Lemma \ref{lem1}, we get
$$\nabla \Phi_\epsilon (\Lambda)=\nabla \Psi_k+  o(1),$$
where $o(1)$ is uniform with respect to parameters $\Lambda$ with $M^{-1}<\Lambda_i<M$, for fixed large $M$.

Next we analyze the critical points of $\Psi_k(\Lambda)$, by writing
$$\Psi_k(\Lambda)= \varphi_1(\Lambda_1)+\sum_{i=2}^k \varphi_i(\Lambda_i),$$
where
\begin{align*}
\varphi_k(s)&=a_5 V(0) s^{p^*-q} + a_3 k \log s,\\
\varphi_i(s)&= (k-i+1) a_3 \log s -a_2 s, \quad \quad i=2,\cdots, k.
\end{align*}
Notice that, $\varphi_i$ has a unique maximum point $\Lambda_i^*=(k-i+1)\frac{a_3}{a_2}$, for $i=2,\cdots, k$. If we further assume that $V(0)<0$, then $\varphi_1(s)$ also has a unique maximum point $\Lambda_1^*=\left[  - \frac{a_3 k}{a_5 V(0) (p^*-q)} \right]^{\frac{1}{p^*-q}}$.

Since the critical point $$\Lambda^*=\left(   \left[ - \frac{a_3 k}{a_5 V(0) (p^*-q)} \right]^{\frac{1}{p^*-q}}, \frac{(k-1)a_3}{a_2}, \cdots, \frac{a_3}{a_2} \right)$$
of $\Psi_k$ is nondegenerate, it follows that the local degree $\deg (\nabla \Psi_k, \mathcal{V}, 0)$ is well defined and nonzero. Here $\mathcal{V}$ denotes a small neighborhood of $\Lambda^*$ in $\mathbb{R}^k$. Hence, $\deg (\nabla \Phi_\epsilon, \mathcal{V}, 0)\neq 0$, if $\epsilon$ is small enough. We conclude that there exists a critical point $\Lambda_\epsilon^*$ of $\Phi_\epsilon$ satisfying
$$\Lambda_\epsilon^*=\Lambda^*+o(1).$$
For $\xi_\epsilon=\xi (\Lambda_\epsilon^*)$, the functions
$$v=\bar{U}+\phi(\xi_\epsilon)$$
are solutions of \eqref{eq4}. From the equation \eqref{eq10} and Proposition \ref{prop3}, we derive that $v=\bar{U}(1+o(1))$. If we set $\xi^*=\xi(\Lambda^*)$, then it is also true that
$$v(x)=\sum_{i=1}^{k}U(x-\xi_i^*)(1+o(1)).$$

Now, changing the variables back, we have that
$$u_\epsilon^*(r)= \gamma_N \sum_{i=1}^k \left(   \frac{1}{1+e^{(p^*-1)\xi_i^*} r^2 }  \right)^{\frac{N-2}{2}} e^{\xi_i^*} (1+o(1)),$$
where $e^{\xi_i^*}=e^{-(i-1)-\frac{1}{p^*-q}} \Pi_{j=1}^i (\Lambda_j^*)^{-1}   $, is a solution of \eqref{eq3}. We conclude that the ansatz given for $v$ provides a spike-tower solution for \eqref{eq2}.

\section{Proof of Theorem \ref{teo2}}

In this section we prove Theorem \ref{teo2}. The proof is very similar to that of Theorem \ref{teo1}, then we just highlight below the most critical changes. Since we are seeking for radial solutions of \eqref{eq1}, we consider again the following slightly supercritical equation
\begin{equation}
\begin{cases}\label{eq17}
	u''(r)+ \frac{N-1}{r} u'(r)+ u^{p^*+\epsilon}(r)- V(r) u^q(r)=0\\
	u(r)\rightarrow 0, \mbox{ as } r \rightarrow \infty\\
\end{cases}
\end{equation}
with $\epsilon>0$, but this tame we take $q>p^*$. We consider the transformation
\begin{equation}\label{EF}
v(x)=r^{\frac{2}{p^*-1}} u(r),\quad \mbox{ with } r=e^{\frac{p^*-1}{2}x},
\end{equation}
for $x\in\mathbb{R}$. Then the problem \eqref{eq17} becomes
\begin{equation}
\begin{cases}\label{eq18}
	v''(x)-v(x)+\beta \left[   e^{-\epsilon x} v^{p^*+\epsilon}(x)- V\left( e^{\frac{p^*-1}{2}x} \right)  e^{(p^*-q)x}  v^q(x)    \right] =0,\\
	0<v(x)\rightarrow 0, \mbox{ as } |x| \rightarrow \infty\\
\end{cases}
\end{equation}
We recall that $\beta=\left( \frac{2}{N-2}    \right)^2$. Again we denote $\omega(x)=V\left( e^{\frac{p^*-1}{2}x} \right)$.

The energy functional related to \eqref{eq18} is
\begin{equation}\label{eq19}
\hat{E}_\epsilon(\psi)=\hat{I}_\epsilon (\psi)+\frac{\beta}{q+1}\int_{\mathbb{R}} \omega(x) e^{(p^*-q)x} |\psi|^{q+1} dx
\end{equation}
where
$$\hat{I}_\epsilon (\psi)=   \frac{1}{2}\int_{\mathbb{R}} (|\psi'|^2 + |\psi|^2) dx -  \frac{\beta}{p^*+\epsilon+1}\int_{\mathbb{R}}  e^{-\epsilon x} |\psi|^{p^*+\epsilon+1} dx.$$
We choose, for small $\epsilon>0$, the points $\xi_i$ as follows:
\begin{align}\label{eq19}
&\hat{\xi}_1= -\frac{1}{q-p^*} \log \epsilon - \log \hat{\Lambda}_1\\
\nonumber \hat{\xi}_{i+1}-\hat{\xi}_i=- &\log \epsilon -\log \hat{\Lambda}_{i+1}, \quad i=1,\cdots, k-1
\end{align}
where the $\hat{\Lambda}_i$'s are positive parameters. We want to find a solution of \eqref{eq18} of the form
$$\hat{v}(x)=\sum_{i=1}^k U(x-\hat{\xi}_i)+\phi(x), $$
where $U$ is defined by \eqref{eq7} and $\phi$ is small. We set $\hat{\Lambda}=(\hat{\Lambda}_1,\hat{\Lambda}_2,\cdots,\hat{\Lambda}_k)$.

In this setting, Lemma \ref{lem1} takes the following form.

\begin{Lemma}\label{lem5}
Let $N\geq 3$, $\delta>0$ fixed, $k\in \mathbb{N}$. Assume that
\begin{equation}\label{eq20}
\delta<\hat{\Lambda}_i <\delta^{-1},\quad i=1,2,\cdots, k.
\end{equation}
Then there exist positive numbers $a_1$, $i=1,\cdots, 4$ and $\hat{a}_5$, depending on $N$, $p$ and $q$, such that
$$\hat{E}_\epsilon (\hat{U}_S)= k a_1 + \epsilon \hat{\Psi}_k(\hat{\Lambda})  +  k\epsilon \beta a_4 + \epsilon \hat{\theta}_\epsilon (\hat{\Lambda}) - \frac{a_3 k}{2(q-p^*)}((1-k)(q-p^*)-2) \epsilon \log \epsilon$$
where
\begin{equation}\label{PSI1}
\hat{\Psi}_k(\hat{\Lambda}) =  a_3 k \log \hat{\Lambda}_1 + \hat{a}_5 V_\infty \Lambda_{1}^{(q-p^*)} + \sum_{i=1}^{k} [(k-i+1) a_3 \log \hat{\Lambda}_i - a_2 \hat{\Lambda}_i  ]
\end{equation}
and $\hat{\theta}_\epsilon(\hat{\Lambda})\rightarrow 0$ as $\epsilon \rightarrow 0$ uniformly in $C^1$-sense on the set of $\hat{\Lambda}_i$'s satisfying \eqref{eq20}. Moreover, the constants $a_i$, $i=1,2,\cdots, 4$ are given as in \eqref{cc1} and $\hat{a}_5$ is defined by
\end{Lemma}
\begin{equation}\label{cc2}
\hat{a}_5=\frac{\beta}{q +1}  \int_{\mathbb{R}}  e^{-(q-p^*)x}   U^{q+1}(x)  dx.
\end{equation}

If we assume that $V_\infty<0$ then $\hat{\Psi}_k$ has a unique nondegenerate critical point, given by
$$\hat{\Lambda}^*=\left(   \left[   \frac{a_3 k}{(p^*-q)\hat{a}_5 V_\infty}  \right]^{\frac{1}{q-p^*}}, \frac{(k-1)a_3}{a_2}, \frac{(k-2)a_3}{a_2}, \cdots, \frac{a_3}{a_2}  \right).$$
The finite-dimensional reduction and the conclusion of the Theorem follows in a similar way to the proof of Theorem \ref{teo1}.

\section*{Acknowledgements}
{\small The first author is supported by FONDECYT Grant 1160135 and Millennium Nucleus Center for Analysis of PDE, NC130017.
The second author was supported by FAPESP (Brazil) Grant \#2016/04925-7.}

\bigskip

\end{document}